\documentclass{article}
\usepackage[english]{babel}
\usepackage{amsmath}
\usepackage{amsthm}
\usepackage{amssymb}
\usepackage{amscd}
\usepackage[all]{xy}
\usepackage[latin1]{inputenc}

\newtheorem{teor}{Theorem}[section]
\newtheorem{defi}{Definition}
\newtheorem{lema}[teor]{Lemma}
\newtheorem{prop}[teor]{Proposition}
\newtheorem{cor}[teor]{Corollary}
\newtheorem{rem}[teor]{Remark}

\newtheorem{quesrem}[teor]{Question and Remark}

\def\Ext {\mathop{\rm Ext}\nolimits}

\begin{document}

\title{Hearts of set-generated t-structures have a set of generators}

\author{Manuel Saor\'in \thanks{The  author is supported by the Grant PID2020-113206GB-I00, funded by MCIN/AEI/10.13039/501100011033, and the project 22004/PI/22, funded by the Fundaci\'on 'S\'eneca' of Murcia, both  with a part of FEDER funds. The paper was started during a visit of the author to the University of Verona in March-2023, within the 'Network on Silting Theory', funded by the Deutsche Forschungsgemeinschaft. Several discussions with Lidia Angeleri-H\"ugel, Rosie Laking, Jorge Vit\'oria and Frederik Marks during that visit were the starting point of the paper. Later on the author has also received many stimulating suggestions from Jan Stovicek.  I thank all these institutions and researchers for their help.  }\\ Departamento de Matem\'aticas\\
Universidad de Murcia, Aptdo. 4021\\
30100 Espinardo, Murcia\\
SPAIN\\ {\it msaorinc@um.es}}

\date{}

\maketitle


\begin{abstract}

{\bf We show that if $\alpha$ is a regular cardinal,  $\mathcal{D}$ is an $\alpha$-compactly generated triangulated category, in the sense of Neeman \cite{N}, and $\tau$ is a t-structure in $\mathcal{D}$ generated by a set of $\alpha$-compact objects, then the heart of $\tau$ is a locally $\alpha$-presentable (not necessarily Ab5) abelian category. This is a generalization of the fact that any compactly generated t-structure (in a compactly generated) triangulated category with coproducts has a heart which is a locally finitely presented Grothendieck category (\cite{SS}) since locally finitely presented abelian categories are Ab5. As a consequence, in a well-generated triangulated category any t-structure generated by a set of objects has a heart with a set of generators.
 }
\end{abstract}


\vspace*{1cm}

{\bf Introduction}

\vspace*{0.3cm}

t-Structures in triangulated categories were introduced by Beilinson, Bernstein and Deligne \cite{BBD} in their study of perverse sheaves on an algebraic or analytic variety. A t-structure in a triangulated category is a pair of subcategories satisfying certain axioms (see Subsection \ref{subsect.triangcats-tstructs}) that guarantee that their intersection has a structure of abelian category and the existence of a cohomological functors from the ambient triangulated category to this heart. This allows a sort of intrinsic cohomology theory where the cohomology spaces are objects of the triangulated category itself. t-Structures are nowadays a common tool in several branches of Mathematicas and Theoretical Physics and they have called the attention of many top mathematicians.

When $\mathcal{D}$ is a triangulated category with coproducts and $\tau$ is a t-structure in $\mathcal{D}$ with heart $\mathcal{H}$, a popular problem in the last two decades has been that of determining when $\mathcal{H}$ is a module or a Grothendieck category. The initial results concentrated mainly in the case when the ambient triangulated category is the (unbounded) derived category of a module or Grothendieck category and the t-structure is tilted from the standard one (see \cite{HKM}, \cite{CGM}, \cite{CMT}, \cite{PS1}, \cite{PS2}), later results studied the case when the t-structure is the one associated to a  (co)tilting module or a (co)silting t-structure (see \cite{St},\cite{AMV},  \cite{B}). For a general $\mathcal{D}$ (always with coproducts)  the efforts were addressed to study conditions on the t-structure that guaratee that its heart is a Grothendieck category, for which most of the time a model for $\mathcal{D}$, as a stable $\infty$-category or a strong stable derivator, were necessary (see \cite{Lurie}, \cite{SSV}, \cite{Laking}). A recent model-free result   in this vein (see \cite[Theorem 8.31]{SS}) states that the heart of any compactly generated t-structure is a locally finitely presented Grothendieck category. 

When determining whether $\mathcal{H}$ is a Grothendieck category, two facts should be checked. Namely, the Ab5 condition and the existence of a generator (or a set of generator) for $\mathcal{H}$. In the papers mentioned above, the hard part has been the Ab5 condition and the existence of a generator was proved either lifting the problem to the model, requiring here some sort of bounded cardinality condition (\cite{Lurie}, \cite{SSV}), or just by applying the associated cohomological functor $H:\mathcal{D}\longrightarrow\mathcal{H}$ to an adequate skeletally small thick subcategory of $\mathcal{D}$, e.g. the subcategory $\mathcal{D}^c$ of compact objects when $\mathcal{D}$ is compactly generated (see \cite{Bo}). 

The goal of this paper is to prove that if $\mathcal{D}$ is a well-generated triangulated category in the sense of Neeman \cite{N}, then any t-structure in $\mathcal{D}$ generated by a set of objects has a heart $\mathcal{H}$ which has a set of generators, even if $\mathcal{H}$ need not be Ab5. So there is no need of a model for $\mathcal{D}$, but instead the well-generation is required, a condition that is general enough as to include most triangulated categories with coproducts appearing in practice. Our initial motivation came from the already mentioned   \cite[Theorem 8.31]{SS}, that may be seen as the $\aleph_0$-version of the following result, an extended version of which is Theorem \ref{teor.main result2}. We refer to Section \ref{sect.preliminaries} for the terminology. 

\vspace*{0.3cm}

{\bf THEOREM  A}: Let $\mathcal{D}$ be a well-generated triangulated category, $\mathcal{X}$ a set of objects and let $\tau =((\mathcal{X}^{\perp_{\leq 0}}),\mathcal{X}^{\perp_{< 0}})$ be the t-structure generated by $\mathcal{X}$. Let $\alpha$ be the smallest of the regular cardinals $\beta$ such that $\mathcal{D}$ is $\beta$-compactly generated and all objects $\mathcal{X}$ are $\beta$-compact. Then the heart $\mathcal{H}_\tau$ of $\tau$ is a locally $\alpha$-presentable (not necessarily Ab5) abelian category.

\vspace*{0.3cm}

As an immediate consequence, one gets:

\vspace*{0.3cm}

{\bf COROLLARY B}: In a well-generated triangulated category any t-structure generated by a set of objects has a heart that has a set of generators.

\vspace*{0.3cm}

The organization of the paper goes as follows. In Section \ref{sect.preliminaries} we introduce essentially all the concepts used through the paper and give some auxiliary results that will be relevant for the proof of the theorem. In Section \ref{sec.coproducts-dirlimits}, which is crucial for the proof, we consider  any additive category with coproducts and show how to conveniently  express in it  coproducts of objects and morphisms between them as suitable direct limits.  In the final Section \ref{sec.The theorem}, the main Theorem \ref{teor.main result2} is stated at the beginning and its proof is developed through the whole section via several auxiliary results.

\section{Preliminaries} \label{sect.preliminaries}

Throughout the paper all categories that appear will be either additive or (always strictly  full) subcategories of additive categories. If $\mathcal{C}$ is such a category and $X$ and $Y$ are objects of $\mathcal{C}$, we shall denote by $\text{Hom}_\mathcal{C}(X,Y)$ the set of morphisms  $X\longrightarrow Y$ and sometimes, for simplification purposes, we will simply put $(X,Y)=\text{Hom}_\mathcal{C}(X,Y)$ or $(-,Y):=\text{Hom}_\mathcal{C}(-,Y):\mathcal{C}^{op}\longrightarrow\text{Ab}$ for the associated representable functor.    When $\mathcal{C}$ is either abelian or triangulated (see Subsection \ref{subsect.triangcats-tstructs}, $\mathcal{X}\subseteq\mathcal{C}$ is a subcategory and  $I\subseteq\mathbb{Z}$ is a subset, we will let $\mathcal{X}^{\perp_{I}}$ be the subcategory that consists of the objects $Y\in\mathcal{C}$ such that $\text{Ext}_\mathcal{C}^i(X,Y)=0$, for all $X\in\mathcal{X}$ and all $i\in I$. As it is customary, in case $\mathcal{C}$ is triangulated $\text{Ext}_\mathcal{C}^i(X,Y):=\text{Hom}_\mathcal{C}(X,Y[i])$.

We will say that an additive category $\mathcal{A}$ \emph{has (co)products} when it has arbitrary set-indexed (co)products. In such case, if $\mathcal{S}$ is any class (resp. set) of objects in $\mathcal{A}$, we will denote by $\text{Coprod}(\mathcal{S})$ the subcategory whose objects are those isomorphic to coproducts of objects in $\mathcal{S}$ and by $\text{Add}(\mathcal{S})$ the one whose objects are the direct summands of objects in $\text{Coprod}(\mathcal{S})$. When $I$ is a directed set, an $I$-\emph{direct (or $I$-directed) system} in $\mathcal{A}$ is just a functor $X:I\longrightarrow\mathcal{A}$, where $I$ is viewed as a category with a unique morphism $i\rightarrow j$ for each pair $(i,j)\in I\times I$ such that $i\leq j$.  Giving such an $X$ is clearly equivalent to giving  a pair families $[(X_i)_{i\in I},(f_{ij})_{i\leq j}]$, where $X_i=X(i)\in\mathcal{A}$ and $f_{ij}\in\text{Hom}_\mathcal{A}(X_i,X_j)$, for all $i,j\in I$ such that $i\leq j$, so that $f_{ii}=1_{X_i}$ and $f_{jk}\circ f_{ij}=f_{ik}$ whenever $i\leq j\leq k$. Frequently, when there is no confusion, we will omit the reference to the $f_{ij}$ and will say simply that $(X_i)_{i\in I}$ is an $I$-direct system.   

Given an additive category $\mathcal{A}$ and a subcategory $\mathcal{S}$, we will say that an object $A\in\mathcal{A}$ is a \emph{direct limit (or filtered colimit) of objects in $\mathcal{S}$} when there is a directed set $I$ and an $I$-direct system $(S_i)_{i\in I}$ in $\mathcal{S}$ such that the direct limit (=filtered colimit) $\varinjlim S_i$ exists in $\mathcal{A}$ and there is an isomorphism $A\cong\varinjlim S_i$. A particular situation will be relevant for us in this paper. When $\alpha$ is a regular cardinal, an \emph{$\alpha$-directed set} is a directed set $I$ such that any subset of cardinality $<\alpha$ has an upper bound in $I$ (note that 'directed set' is synonymous of '$\aleph_0$-directed set') . We will say that  $A$ is an  \emph{$\alpha$-directed colimit} of objects in $\mathcal{S}$ when there is an $\alpha$-directed set $I$ and an $I$-direct system  $(S_i)_{i\in I}$ in $\mathcal{S}$ such that $\varinjlim S_i$ exists in $\mathcal{A}$ and there is an isomorphism $A\cong\varinjlim S_i$. We will say that $\mathcal{A}$ \emph{has $\alpha$-directed colimits} when all $\alpha$-direct systems in $\mathcal{A}$ have a direct limit.

\subsection{$\alpha$-small and $\alpha$-presentable objects} 

All through this subsection $\mathcal{A}$ is an additive category that has coproducts and $\alpha$ is a regular cardinal.

We say that $X\in\mathcal{A}$ is an \emph{$\alpha$-small object} when, for each family $(A_i)_{i\in I}$ of objects of $\mathcal{A}$ and each morphism $f:X\longrightarrow\coprod_{i\in I}A_i$, there is a subset $J=J(f)\subseteq I$ such that $|J|<\alpha$, were $|-|$ denotes the cardinality of a set, and $f$ factors as a composition $f:X\longrightarrow\coprod_{j\in J}A_j\stackrel{\iota_J}{\longrightarrow}\coprod_{i\in I}A_i$, where $\iota_J$ is the canonical section. Note that $X$ is $\aleph_0$-small (=\emph{small}) exactly when the functor $\text{Hom}_\mathcal{A}(X,-):\mathcal{A}\longrightarrow\text{Ab}$ preserves coproducts. We will denote by $\text{Small}_\alpha (\mathcal{A})$ the subcategory of $\alpha$-small objects, which is closed under coproducts of $<\alpha$ objects.  

When we also assume that $\mathcal{A}$ has $\alpha$-directed colimits, an object $X\in\mathcal{A}$ is called \emph{$\alpha$-presentable} when, for each $\alpha$-directed system $(Y_i)_{i\in I}$, the canonical morphism $\varinjlim\text{Hom}_\mathcal{A}(X,Y_i)\longrightarrow\text{Hom}_\mathcal{A}(X,\varinjlim Y_i)$ is an isomorphism. We denote by $\text{Pres}^{<\alpha}(\mathcal{A})$ the subcategory of $\alpha$-presentable objects, which is closed under coproducts of $<\alpha$ objects and cokernels, when they exist. 
We will say that $\mathcal{A}$ is a \emph{locally $\alpha$-presentable category} when $\text{Pres}^{<\alpha}(\mathcal{A})$  is skeletally small and each object of $\mathcal{A}$ is an $\alpha$-directed colimit of $\alpha$-presentable objects.  It is not hard to see that, in case $\mathcal{A}$ is abelian, it is  locally $\alpha$-presentable  if and only if it has a set of generators which are $\alpha$-presentable.  $\aleph_0$-presentable objects are usually called \emph{finitely presented} and a locally $\aleph_0$-presentable category is simply called \emph{locally finitely presented}. The reader is referred to \cite{AR} for a more general version of local $\alpha$-presentability that does not require the category $\mathcal{A}$ to be additive. 

Already at this step it is convenient to give the following result, that will be relevant later on in the paper. 

\begin{prop} \label{prop.projecti-small-objects}
Let $\mathcal{A}$ be a additive category with arbitrary coproducts and $\alpha$-directed colimits, where $\alpha$ is a regular cardinal, and let $P$ be an object of $\mathcal{A}$. Consider the following assertions:

\begin{enumerate}
\item $P$ is $\alpha$-presentable.
\item $P$ is $\alpha$-small.
\end{enumerate}
The implication $(1)\Longrightarrow (2)$ always holds. When $P$ is a projective object, both assertions are equivalent. 
\end{prop}
\begin{proof}
$(1)\Longrightarrow (2)$ Let $(X_i)_{i\in I}$ be a family of objects in $\mathcal{A}$ and $f:P\longrightarrow\coprod_{i\in I}X_i$. Now we can apply Proposition \ref{prop.coproduct-as-directlimit} below with $\mathcal{J}=\mathcal{P}_\alpha (I)$ the set of subsets $J\subseteq I$ such that $|J|<\alpha$. Then $\coprod_{i\in I}X_i=\varinjlim_{J\in\mathcal{P}_\alpha (I)}X_J$, where $X_J:=\coprod_{j\in J}X_j$ for all $J\in\mathcal{P}_\alpha (I)$. The $\alpha$-presentability of $P$ tells us that the morphism  $f:P\longrightarrow\coprod_{i\in I}X_i=\varinjlim_{J\in\mathcal{P}_\alpha (I)}X_J$ factors in the form $P\stackrel{\tilde{f}}{\longrightarrow}X_J\stackrel{\iota_J}{\hookrightarrow}\coprod_{i\in I}X_i$, where $\iota_J$ is the canonical section, for some $J\in\mathcal{P}_\alpha (I)$. That is, $f$ factors through a subcoproduct $X_J=\coprod_{j\in J}X_j$, where $|J|<\alpha$. Therefore $P$ is $\alpha$-small. 

\vspace*{0.3cm}

$(2)\Longrightarrow (1)$ Let $[(Y_i)_{i\in I},(u_{ij}:Y_i\longrightarrow Y_j)_{i\leq j}]$ be an $\alpha$-direct system. By the explicit construction of the direct limit, we know that there is a cokernel sequence $\coprod_{i<j}Y_{ij}\stackrel{\varphi}{\longrightarrow}\coprod_{i\in I}Y_i\stackrel{p}{\longrightarrow}\varinjlim Y_i\rightarrow 0$, i.e. $p$ is the cokernel map of $\varphi$. Let now $f:P\longrightarrow\varinjlim Y_i$ be any morphism. By projectivity of $P$, we have a morphism $g:P\longrightarrow\coprod_{i\in I}X_i$ such that $p\circ g=f$. Note also that if $\iota_k:Y_k\longrightarrow\coprod_{i\in I}Y_i$ is the canonical section, then $u_k:=p\circ\iota_k:Y_k\longrightarrow\varinjlim Y_i$ is the canonical map to the direct limit. By the $\alpha$-smallness of $P$, the morphism $g$ factors in the form $P\stackrel{\tilde{g}}{\longrightarrow}Y_J=\coprod_{j\in J}Y_i\stackrel{\iota_J}{\hookrightarrow}\coprod_{i\in I}Y_i$, for some $J\in\mathcal{P}_\alpha (I)$. For each $j'\in J$, we shall denote by $\iota'_{j'}:Y_{j'}\longrightarrow\coprod_{j\in J}Y_j$ the canonical map to the $J$-coproduct, so that $\iota_J\circ\iota'_{j'}=\iota_{j'}$. Then we have $f=p\circ g=p\circ\iota_J\circ\tilde{g}$. Since $I$ is an $\alpha$-directed set we can pick up a $m\in I$ such that $j\leq m$, for all $j\in J$. The maps $u_{jm}:Y_j\longrightarrow Y_m$ ($j\in J$) induce a morphism $u:\coprod_{j\in J}Y_j\longrightarrow Y_m$ such that $u\circ\iota'_j=u_{jm}$, for all $j\in J$. It follows that $u_m\circ u\circ\iota'_j=u_m\circ u_{jm}=u_j=p\circ\iota_j=p\circ\iota_J\circ\iota'_j$, for all $j\in J$. It then follows that $u_m\circ u=p\circ\iota_J$, which in turn implies that $u_m\circ u\circ\tilde{g}=p\circ\iota_J\circ\tilde{g}=p\circ g=f$. In particular $f$ factors through the map $u_m:Y_m\longrightarrow\varinjlim Y_i$. This implies that $f$ is in the image of the canonical morphism $\kappa:\varinjlim\text{Hom}_\mathcal{A}(P,Y_i)\longrightarrow\text{Hom}_\mathcal{A}(P,\varinjlim Y_i)$. Hence this latter map is an epimorphism. 

We next prove that $\kappa$ is a monomorphism.  By properties of direct limits in $\text{Ab}$, we know that any element of $\varinjlim\text{Hom}_\mathcal{A}(P,Y_i)$ is of the form $(u_k)_*(f)=f\circ u_k=f\circ p\circ\iota_k$, for some $k\in I$ and  some morphism $f:P\longrightarrow Y_k$ in $\mathcal{A}$. 
Suppose that $(u_k)_*(f)\in\text{Ker}(\kappa)$, which  is equivalent to say that  the composition $P\stackrel{f}{\longrightarrow}Y_k\stackrel{\iota_k}{\longrightarrow}\coprod_{i\in I}Y_i$ factors through $\varphi$. Let then fix a morphism $g:P\longrightarrow\coprod_{i\leq j}Y_{ij}$ such that $\varphi\circ g=\iota_k\circ f$. Due to the the $\alpha$-smallness of $P$, we can select a subset $\tilde{\Lambda}\subseteq\{(i,j)\in I\times I\text{: }i\leq j\}$ such that $|\tilde{\Lambda}|<\alpha$ and $g$ factors through the inclusion $\coprod_{(i,j)\in\tilde{\Lambda}}Y_{ij}\hookrightarrow\coprod_{i\leq j}Y_{ij}$. Associated to $\tilde{\Lambda}$ and using that $I$ is $\alpha$-directed,  we can easily choose a subset $\Lambda\subseteq I$ such that:  (i) $| \Lambda|<\alpha$; (ii) all first and second components of the $(i,j)\in\tilde{\Lambda}$ belong to $\Lambda$; (iii) $k\in\Lambda$; (iv) $\Lambda$ has a maximum $\mu =\text{max}(\Lambda)$. We obviously have that $\tilde{\Lambda}\subseteq\{(\lambda, \lambda')\in\Lambda\times\Lambda\text{: }\lambda\leq\lambda'\}$. We then get the following  diagram, where the hook arrows are the obvious sections, $g=\iota\circ\tilde{g}$, $\varphi_\Lambda$ is the canonical morphism whose cokernel is $\varinjlim Y_\lambda$ and the exterior diagram is commutative, i.e. $\varphi\circ\iota\circ\tilde{g}=\varphi\circ g=\iota_k\circ f=\iota_\Lambda^I\circ\iota_k^\Lambda\circ f$. 

$$\xymatrix{ & P \ar[d]^f \ar[ddl]_{\tilde{g}} \\ & Y_k  \ar@{^(->}[d]^{\iota^{\Lambda}_k}\\ \coprod_{\lambda \leq \lambda'} Y_{\lambda \lambda'} \ar@{^(->}[d]^{\iota}\ar[r]^{\varphi_{\Lambda}} &  \coprod_{\Lambda} Y_\lambda \ar@{^(->}[d]^{\iota^{I}_{\Lambda}} \\ \coprod_{i \leq j}Y_{ij} \ar[r]^{\varphi} & \coprod_{I} Y_i}$$

Due to the monomorphic condition of $\iota$ and $\iota_\Lambda^I$, the upper triangle is commutative, i.e. $\varphi_\Lambda\circ\tilde{g}=\iota_k^\Lambda\circ f$. But this means that   $f$  is mapped by the canonical morphism $v_k:\text{Hom}_\mathcal{A}(P,Y_k)\longrightarrow\varinjlim_\Lambda\text{Hom}_\mathcal{A}(P,Y_\lambda)$ onto an element which is in the kernel of the canonical map $\kappa_\Lambda: \varinjlim_\Lambda\text{Hom}_\mathcal{A}(P,Y_\lambda)\longrightarrow\text{Hom}_\mathcal{A}(P,\varinjlim_\Lambda Y_\lambda)$. Since $\mu =\text{max}(\Lambda )$ the domain and codomain of $\kappa_\Lambda$ are both isomorphic to $\text{Hom}_\mathcal{A}(P,Y_\mu)$, from which we readily get that $\kappa_\Lambda$ is an isomorphism. It follows that $0=v_k(f)\cong u_{k\mu}\circ f$. We have then found an index $\mu\in I$, $\mu\geq k$, such that $u_{k\mu}\circ f=0$. This implies that $(u_k)_*(f)=u_k\circ f=u_\mu\circ u_{k\mu}\circ f=0$. Therefore $\kappa$ is a monomorphism, and hence an isomorphism. 
\end{proof}

Recall that if $\mathcal{C}$ is any additive category, then a functor $M:\mathcal{C}^{op}\longrightarrow\text{Ab}$ is said to be a \emph{finitely presented functor (or finitely presented right $\mathcal{C}$-module)} when there is an exact sequence of functors $(-,C')\stackrel{(-,\alpha)}{\longrightarrow}(-,C)\longrightarrow M\rightarrow 0$, where $(-,X):=\text{Hom}_\mathcal{C}(-,X)$ for each object $X\in\mathcal{C}$. We denote by $\text{mod}-\mathcal{C}$ the category of finitely presented functors, which is an additive category with cokernels. The Yoneda functor $Y_\mathcal{C}:\mathcal{C}\longrightarrow\text{mod}-\mathcal{C}$ ($C\rightsquigarrow (-,C)$) is fully faithful and its essential image is a class of projective generators of $\text{mod}-\mathcal{C}$. In the following result we summarize well-known results that will be needed later on, refering to the references given in the proof for the non-introduced terminology:

\begin{prop} \label{prop.C-modules}
Let $\mathcal{C}$ be an additive category with coproducts and $\alpha$ a regular cardinal. The following assertions hold:

\begin{enumerate}
\item $\text{mod}-\mathcal{C}$ has coproducts (and hence it is a cocomplete category) and  $Y_\mathcal{C}$ preserves coproducts, i.e. for any family $(C_i)_{i\in I}$ of objects of $\mathcal{C}$ the representable functor $(-,\coprod_{i\in I}C_i)$ is the coproduct of the $(-,C_i)$ in $\text{mod}-\mathcal{C}$.
\item $Y_\mathcal{C}=(-,C)$ is $\alpha$-small (equivalently $\alpha$-presentable) in $\text{mod}-\mathcal{C}$ if and only if $C$ is an $\alpha$-small object of $\mathcal{C}$. 
\item If $\mathcal{C}$ has weak kernels, then $\text{mod}-\mathcal{C}$ is abelian. In particular, if there is a set $\mathcal{S}\subseteq\mathcal{C}$ such that $\mathcal{C}=\text{Coprod}(\mathcal{S})$ then $\text{mod}-\mathcal{C}$ is abelian and $\{(-,S)\text{: }S\in\mathcal{S}\}$ is a set of projective generators of $\text{mod}-\mathcal{C}$.
\end{enumerate}
\end{prop}
\begin{proof}
Assertion 1 and the initial part of assertion 3 are in \cite[Lemma 1]{Kr3}. Moreover if  $\mathcal{C}=\text{Coprod}(\mathcal{S})$, for some set $\mathcal{S}$, then $\mathcal{C}$ has weak kernels (see \cite[Lemma 2.4]{SS}). Assertion 2 is  \cite[Lemma 1]{Kr2} and can be also easily deduced from Proposition \ref{prop.projecti-small-objects}.
\end{proof}

\subsection{Preservation of $\alpha$-presentable objects}

The following result will be relevant later on in the paper.

\begin{prop} \label{prop.preservation-alpha-presented}
Let $(F:\mathcal{A}\longrightarrow\mathcal{B},G:\mathcal{B}\longrightarrow\mathcal{A})$ an adjoint pair of additive functors between cocomplete abelian categories, where $\mathcal{A}$ is supposed to be locally $\alpha$-presentable, for a given regular cardinal $\alpha$. The following assertions are equivalent:

\begin{enumerate}
\item $F$ preserves $\alpha$-presentable objects.
\item $G$ preserves $\alpha$-directed colimits.
\end{enumerate}
\end{prop}
\begin{proof}

Let $Q\in\mathcal{A}$ be any object and  let  $(Y_i)_{i\in I}$ be any $\alpha$-directed system in $\mathcal{B}$. We let  $\psi: \varinjlim G(Y_i)\longrightarrow G(\varinjlim Y_i)$  be the canonical morphism and $\kappa:\varinjlim (-,G(Y_i))\longrightarrow(-,\varinjlim G(Y_i))$ the canonical natural transformation. We have the following sequence of functors $\mathcal{A}^{op}\longrightarrow\text{Ab}$ and natural transformations between them, whose composition is the canonical natural transformation $\varinjlim (F(-),Y_i)\longrightarrow (F(-),\varinjlim Y_i)$:
 
\begin{center}
$\varinjlim (F(-),Y_i)\stackrel{\cong}{\longrightarrow}\varinjlim (-,G(Y_i))\stackrel{\kappa}{\longrightarrow}(-,\varinjlim G(Y_i))\stackrel{\psi_*}{\longrightarrow}(-,G(\varinjlim Y_i))\stackrel{\cong}{\longrightarrow}(F(-),\varinjlim Y_i)$.
\end{center}

Note that, by adjunction, the first and fourth arrows are natural isomorphisms. 

\vspace*{0.3cm}

$(1)\Longrightarrow (2)$ Evaluate the last sequence at an $\alpha$-presentable object $Q$. Since $F(Q)$ is $\alpha$-presentable as well, we conclude that $\kappa_Q$ is an isomorphism and the composition of the four morphisms is also an isomorphism. Then $(\psi_*)_Q$ is an isomorphism. But since $\mathcal{A}$ is assumed to be locally $\alpha$-presentable, we get that  $(\psi_*)_Q$ is an isomorphism, for all $Q$ in a set of generators of $\mathcal{A}$. It follows that $\psi$ is an isomorphism. 



\vspace*{0.3cm}

$(2)\Longrightarrow (1)$ In this case $\psi_*$ is a natural isomorphism and, for each $\alpha$-presentable object $Q$, the map $\kappa_Q$ is also an isomorphism. It then follows that, when evaluating at any such $Q$, the canonical morphism $\varinjlim (F(Q),Y_i)\longrightarrow (F(Q),\varinjlim Y_i)$ is always an isomorphism. That is, $F(Q)$ is an $\alpha$-presentable object.

\end{proof}

\subsection{On Serre quotient functors with right adjoints in cocomplete abelian categories}

In this subsection we shall see that some known facts about localisation of Grothendieck categories hold in the more general setting of cocomplete abelian categories. We refer to \cite[Subsection 2.2]{SS} for the terminology used in this subsection. 

\begin{defi}
Let $\mathcal{A}$ be any abelian category. A \emph{torsion pair} or \emph{torsion theory} in it is a pair $\mathbf{t}=(\mathcal{T},\mathcal{F})$ of subcategories such that the following two conditions hold: 

\begin{enumerate}
\item[t1)] $\text{Hom}_\mathcal{A}(T,F)=0$ for all $T\in\mathcal{T}$ and $F\in\mathcal{F}$;
\item[t2)] For each object $A\in\mathcal{A}$, there is an exact sequence $0\rightarrow T_A\longrightarrow A\longrightarrow F_A\rightarrow 0$, where $T_A\in\mathcal{T}$ and $F_A\in\mathcal{F}$, that we shall call the \emph{torsion sequence} associated to $A$.
\end{enumerate}

A \emph{torsion (resp. torsionfree) class} in $\mathcal{A}$ is a subcategory $\mathcal{T}$ (resp. $\mathcal{F}$) that is the first (resp. second) component of a torsion pair.
\end{defi}

It is well-known that a torsion class is closed under quotients, extensions and all coproducts that exist in $\mathcal{A}$ and, dually, a torsionfree class is closed under subobjects, extensions and products.  It is also well-known that the $T_A$ and $F_A$ of condition (t2) are uniquely determined, up to isomorphism, and induce functors $t:\mathcal{A}\longrightarrow\mathcal{T}$ and $(1:t):\mathcal{A}\longrightarrow\mathcal{F}$ which are, respectively, right and left adjoint to the inclusion functor. One usually calls $t(A)=T_A$ and $(1:t)(A)=F_A$ the \emph{torsion subobject} and the \emph{torsionfree quotient} of $A$, respectively, with respect to $\mathbf{t}$.

The  torsion class $\mathcal{T}$ or the torsion pair $\mathbf{t}$ are \emph{hereditary} when $\mathcal{T}$ is closed under subobjects.

\begin{prop} \label{prop.Serre quotient functor}
Let $F:\mathcal{A}\longrightarrow\mathcal{B}$ be a Serre quotient functor that has a fully faithful right adjoint $G:\mathcal{B}\longrightarrow\mathcal{A}$ and suppose that $\mathcal{A}$ is cocomplete. The following assertions hold:

\begin{enumerate}
\item If $\mu :1_\mathcal{A}\longrightarrow G\circ F$ is the unit of the adjunction, then $\text{Ker}(\mu_A)$ and $\text{Coker}(\mu_A)$ are objects of  $\mathcal{T}:=\text{Ker}(F)$, for all $A\in\mathcal{A}$.
\item $\mathcal{T}$ is a hereditary torsion class in $\mathcal{A}$ and, for each $A\in\mathcal{A}$, the associated torsion sequence is $0\rightarrow\text{Ker}(\mu_A)\longrightarrow A\longrightarrow\text{Im}(\mu_A)\rightarrow 0$.
\item If $\mathcal{T}_0\subseteq\mathcal{T}$ is any subcategory such that $\mathcal{T}=\text{Gen}(\mathcal{T}_0)$, then $\text{Im}(G)=\mathcal{T}_0^{\perp_{0,1}}$.
\item If $\mathcal{X}$ is a set of generators of $\mathcal{A}$, then the set $\mathcal{T}_0$ of objects of $\mathcal{T}$ which are epimorphic image of objects of $\mathcal{X}$ satisfies that $\mathcal{T}=\text{Gen}(\mathcal{T}_0)$.
\end{enumerate}
\end{prop}
\begin{proof}
Note that, by the fully faithful condition of $G$, the counit $\epsilon :F\circ G\longrightarrow 1_\mathcal{B}$ is a natural isomorphism (see the dual of \cite[Proposition 7.5]{HS}), so that $F$ is a dense functor. In particular,  $\mathcal{B}$ has coproducts (and hence  is also cocomplete) since $F$ preserves them. 

1) This assertion follows from the exactness of $F$ and the adjunction equations. 

2) By adjunction, we have that $\text{Hom}_\mathcal{A}(T,G(B))\cong\text{Hom}_\mathcal{B}(F(T),B)=0$, for all $T\in\mathcal{T}$ and $B\in\mathcal{B}$. This proves that $\text{Im}(G)\subseteq\mathcal{F}:=\mathcal{T}^{\perp_0}$, from which it readily follows that $(\mathcal{T},\mathcal{F})$ is a torsion pair with $0\rightarrow\text{Ker}(\mu_A)\longrightarrow A\longrightarrow\text{Im}(\mu_A)\rightarrow 0$ as torsion sequence associated to $A$, for all $A\in\mathcal{A}$. Exactness of $F$ also proves that $\mathcal{T}$ is closed under taking subobjects, so that $\mathcal{T}$ is a hereditary torsion class. 

3) Note that, by properties of adjunction,  $\text{Im}(G)$ consists of the objects $Y\in\mathcal{A}$ such that $\mu_Y:Y\longrightarrow (G\circ F)(Y)$ is an isomorphism, and we have already seen that $\text{Im}(G)\subseteq\mathcal{T}^{\perp_0}$. 

One readily sees that $\mathcal{T}_0^{\perp_0}=\mathcal{T}^{\perp_0}$. We claim that we also have that  $\mathcal{T}_0^{\perp_{0,1}}=\mathcal{T}^{\perp_{0,1}}$, for which we only need to prove the inclusion $\subseteq$. Let $Y\in\mathcal{T}_0^{\perp_{0,1}}$ and let $T\in\mathcal{T}$ be arbitrary. We need to prove that $\Ext_\mathcal{A}^1(T,Y)=0$. Consider an exact sequence $0\rightarrow T'\longrightarrow\coprod_{i\in I}T_i\longrightarrow T\rightarrow 0$, where $T_i\in\mathcal{T}_0$ for all $i\in I$. Note that, by assertion 2, we also have that $T'\in\mathcal{T}$. Applying the exact sequence of $\text{Ext}_\mathcal{A}(-,Y)$, we get an exact sequence $$0=\text{Hom}_\mathcal{A}(T',Y)\longrightarrow\text{Ext}^1_\mathcal{A}(T,Y)\longrightarrow\text{Ext}^1_\mathcal{A}(\coprod_{i\in I}T_i,Y),$$ and we have a monomorphism $\text{Ext}^1_\mathcal{A}(\coprod_{i\in I}T_i,Y)\rightarrowtail\prod_{i\in I}\text{Ext}_\mathcal{A}^1(T_i,Y)=0$ (see \cite[Lemma 2.5]{PSV1}). 

It remains to prove the equality $\text{Im}(G)=\mathcal{T}^{\perp_{0,1}}$. We start with the inclusion $\subseteq$, for which, due to the proof of assertion 2, we just need to prove that $\text{Im}(G)\subseteq\mathcal{T}^{\perp_1}$.
 Let us consider then an exact sequence $0\rightarrow G(B)\stackrel{u}{\longrightarrow}A\longrightarrow T\rightarrow 0$ (*) in $\mathcal{A}$, where $T\in\mathcal{T}$ and $B\in\mathcal{B}$. By exactness of $F$ and the definition of $\mathcal{T}$,  we get that $F(u):(F\circ G)(B)\longrightarrow F(A)$ is an isomorphism. By using the unit $\mu$, we then get a commutative diagram 

$$\xymatrix{G(B) \ar[d]^{\cong}_{\mu_{G(B)}} \ar[rr]^{u}&& A \ar[d]^{\mu_A}\\ (G \circ F \circ G)(B) \ar[rr]^{\cong}_{(G \circ F)(u)} & & (G\circ F)(A)}$$

from which we get that $\mu_A\circ u$ is an isomorphism, so that $u$ is a section and the sequence (*) splits. This proves that $\text{Im}(G)\subseteq\mathcal{T}_0^{\perp_{0,1}}$.

Conversely, let $Y\in\mathcal{T}^{\perp_{0,1}}$ and consider the map $\mu_Y:Y\longrightarrow (G\circ F)(Y)$. By assertion 1 we have that $\text{Ker}(\mu_Y)=0$ because the inclusion $\text{Ker}(\mu_Y)\hookrightarrow Y$ is the zero morphism. But then we have an exact sequence $0\rightarrow Y\stackrel{\mu_Y}{\longrightarrow}(G\circ F)(Y)\stackrel{p}{\longrightarrow} T\rightarrow 0$, where $T:=\text{Coker}(\mu_Y)\in\mathcal{T}$ by assertion 1. This exact sequence splits and, taking a section $s:T\longrightarrow (G\circ F)(Y)$, we conclude that $T=0$ since we have already proved that $\text{Im}(G)\subseteq\mathcal{T}^{\perp_{0}}$. Therefore $\mu_Y$ is an isomorphism and so $Y\in\text{Im}(G)$. 
\end{proof}

\subsection{Triangulated categories and t-structures} \label{subsect.triangcats-tstructs}

A \emph{triangulated category} $\mathcal{D}$ is an additive category with a self-equivalence $?[1]:\mathcal{D}\longrightarrow\mathcal{D}$ ($X\rightsquigarrow X[1]$) and a distiguished class of sequences $X\stackrel{f}{\longrightarrow} Y\stackrel{g}{\longrightarrow}Z\stackrel{h}{\longrightarrow}X[1]$, called \emph{triangles}, satisfying certain axioms (see [N] for the details), where Verdier's TRIV axiom ( see \cite[D\'efinition 1.1.1]{V}), usually called the \emph{octahedrom axiom (TR4)}, is Neeman's  Proposition 1.4.6. All through this subsection $\mathcal{D}$ is a triangulated category. We will put $?[n]=(?[1])^n$ for each $n\in\mathbb{Z}$. Without mention to the morphisms, we will write indistinctly triangles as above, as $X\longrightarrow Y\longrightarrow Z\longrightarrow X[1]$, as  $X\longrightarrow Y\longrightarrow Z\stackrel{+}{\longrightarrow} $ or even  $Z[-1]\longrightarrow X\longrightarrow Y\longrightarrow Z$.  It follows from the axioms that if $f:X\longrightarrow Y$ is a morphism in $\mathcal{D}$, then it can be completed to a triangle $X\stackrel{f}{\longrightarrow} Y\longrightarrow Z\stackrel{+}{\longrightarrow} $   and also to a triangle $Z'\longrightarrow X\stackrel{f}{\longrightarrow} Y\stackrel{+}{\longrightarrow}$. The objects $Z$ and $Z'$ are uniquely determined by $f$ up to (nonunique) isomorphisms, and will be called here the \emph{cone} and the \emph{cocone} of $f$, denoted $\text{cone}(f)$ and $\text{cocone} (f)$, respectively.  When $\mathcal{A}$ is an abelian category, a functor $H:\mathcal{D}\longrightarrow\mathcal{A}$ is a \emph{cohomological functor} when it takes any triangle $X\longrightarrow Y\longrightarrow Z\stackrel{+}{\longrightarrow} $ to an exact sequence $H(X)\longrightarrow H(Y)\longrightarrow H(Z)$. Due to the axioms TR0-TR4 (see [N]), this yields a long exact sequence $...H^{i-1}(Z)\longrightarrow H^i(X)\longrightarrow H^i(Y)\longrightarrow H^i(Z)\longrightarrow H^{i+1}(X)...$, where $H^{i}:=H\circ (?[i])$ for all $i\in\mathbb{Z}$. 

Given two subcategories $\mathcal{X}$ and $\mathcal{Y}$ of the triangulated category $\mathcal{D}$, we denote by $\mathcal{X}\star\mathcal{Y}$ the subcategory that consists of all objects $D\in\mathcal{D}$ that appear in some triangle $X\longrightarrow D\longrightarrow Y\stackrel{+}{\longrightarrow}$, where $X\in\mathcal{X}$ and $Y\in\mathcal{Y}$. The operation $\star$ is associative (see \cite{BBD}), i.e. $(\mathcal{X}\star\mathcal{Y})\star\mathcal{Z}=\mathcal{X}\star (\mathcal{Y}\star\mathcal{Z})$, for any three subcategories $\mathcal{X}$, $\mathcal{Y}$ and $\mathcal{Z}$. We put $\mathcal{X}^{\star n}=\mathcal{X}\star\stackrel{n+1}{.....}\star\mathcal{X}$, for each $n\in\mathbb{N}$ (so, e.g.,  $\mathcal{X}^{\star 0}=\mathcal{X}$ and $\mathcal{X}^{\star 1}=\mathcal{X}\star\mathcal{X}$). The objects of $\mathcal{X}^{\star n}$ are called \emph{$n$-fold extensions of objects of $\mathcal{X}$}. 

When $\mathcal{D}$ has countable coproducts and $$D_0\stackrel{f_1}{\longrightarrow}D_1\stackrel{f_2}{\longrightarrow}\cdots\stackrel{f_n}{\longrightarrow}D_n\stackrel{f_{n+1}}{\longrightarrow}$$ is a sequence of morphisms, we denote by $1-f$ the morphism $\coprod_{n\in\mathbb{N}}D_n\longrightarrow\coprod_{n\in\mathbb{N}}D_n$ whose $k$-th component $D_k\longrightarrow\coprod_{n\in\mathbb{N}}D_n$ is the composition $D_k\stackrel{\begin{pmatrix} 1_{D_k}\\ -f_{k+1} \end{pmatrix}}{\longrightarrow}D_k\coprod D_{k+1}\stackrel{inclus.}{\hookrightarrow}\coprod_{n\in\mathbb{N}}D_n$, for each $k\in\mathbb{N}$. Associated to the sequence, we  have the so-called \emph{Milnor triangle} $$\coprod_{n\in\mathbb{N}}D_n\stackrel{1-f}{\longrightarrow}\coprod_{n\in\mathbb{N}}D_n\longrightarrow\text{Mcolim}D_n\stackrel{+}{\longrightarrow}$$ and the object $\text{Mcolim}D_n=\text{cone}(1-f)$ is called the \emph{Milnor colimit} of the given sequence. 

A special type of triangulated categories is of special interest in this paper. We assume in this paragraph that $\mathcal{D}$ is a triangulated category with coproducts and that $\alpha$ is a regular cardinal.  We refer the reader to \cite[Chapter 3]{N} for the definition of  \emph{$\alpha$-perfect class (resp. set) of objects} in $\mathcal{D}$.
We say that $\mathcal{S}$ is a \emph{class (resp. set) of $\alpha$-perfect generators} of $\mathcal{D}$ when   $\mathcal{S}^{\perp_\mathbb{Z}}=0$ and $\mathcal{S}$ is $\alpha$-perfect. The category $\mathcal{D}$ is called \emph{$\alpha$-compactly generated} when it has an $\alpha$-perfect \underline{set} of generators that consists of $\alpha$-small objects (see \cite[Definition 8.1.6]{N}). We say that $\mathcal{D}$ is \emph{well-generated} when it is $\beta$-compactly generated, for some regular cardinal $\beta$. In this latter case, for any regular cardinal $\gamma$ one considers the unique maximal $\gamma$-perfect class, denoted $\mathcal{D}^\gamma$, that consists of $\gamma$-small objects. The objects of $\mathcal{D}^\gamma$ are called \emph{$\gamma$-compact}.  When $\mathcal{D}$ is well-generated, each  $\mathcal{D}^\gamma$ is a skeletally small thick subcategory of $\mathcal{D}$ closed under taking coproducts of $<\gamma$ objects,  and moreover $\mathcal{D}=\bigcup_\gamma\mathcal{D}^\gamma$, where $\gamma$ runs trough the class of regular cardinal (see \cite{Kr1}). In particular, any set of objects of $\mathcal{D}$ is contained in some $\mathcal{D}^\gamma$.

Recall (see [BBD]) that a \emph{t-structure} in the triangulated category $\mathcal{D}$ is a pair $\tau =(\mathcal{U},\mathcal{V})$ of subcategories that satisfy the following properties:

\begin{enumerate}
\item[t-S1] $\text{Hom}_\mathcal{D}(U,V[-1])=0$, for all $U\in\mathcal{U}$ and $V\in\mathcal{V}$;
\item[t-S2] $\mathcal{U}[1]\subseteq\mathcal{U}$ (or $\mathcal{V}[-1]\subseteq\mathcal{V}$);
\item[t-S3] $\mathcal{U}\star (\mathcal{V}[-1])=\mathcal{D}$.
\end{enumerate}
In such case $\mathcal{U}$ and $\mathcal{V}$ are called the \emph{aisle} and \emph{coaisle} of the t-structure, respectively, and the intersection $\mathcal{H}:=\mathcal{U}\cap\mathcal{V}$, which is called the \emph{heart} of $\tau$, has a natural structure of abelian category where the short exact sequences 'are' the triangles in $\mathcal{D}$ with their three vertices in $\mathcal{H}$. To this abelian category it is naturally associated a cohomological functor $H_\tau^0:\mathcal{D}\longrightarrow\mathcal{H}$. We will call it \emph{the cohomological functor associated to $\tau$}.

Given a class $\mathcal{S}$ of objects of $\mathcal{D}$, the pair $\tau_\mathcal{S}=( ^\perp(\mathcal{S}^{\perp_{\leq 0}}),\mathcal{S}^{\perp_{< 0}})$ satisfies condition t-S1 and t-S2. When it also satisfies t-S3 and hence $\tau_\mathcal{S}$ is a t-structure, we call it the \emph{t-structure generated by $\mathcal{S}$}. When  $\mathcal{D}$ has coproducts and  $\mathcal{S}$ is a \underline{set} of objects,   it is frequently the case that $\tau_\mathcal{S}$ is a t-structure.  For example, by the results of \cite{N2} (see Proposition \ref{prop.Neeman} below), if $\mathcal{D}$ is well-generated then, for any set of objects $\mathcal{S}$, the pair $\tau_\mathcal{S}$ is a t-structure. Even more, Neeman proves that in that case $\mathcal{U}_\mathcal{S}:= ^\perp(\mathcal{S}^{\perp_{\leq 0}})$ is the smallest subcategory of $\mathcal{D}$ which contains $\mathcal{S}$ and is closed under extensions, coproducts and positive shifts.   A t-structure $\tau$ is said to be \emph{generated by a set} when $\tau =\tau_\mathcal{S}$, for some set of objects $\mathcal{S}$. 


\section{Coproducts as direct limits} \label{sec.coproducts-dirlimits}

All throughout the section $\mathcal{A}$ will be an additive category with coproducts. Suppose that $(X_i)_{i\in I}$ a family of objects in $\mathcal{A}$. For any subset $J\subseteq I$, we will put $X_J:=\coprod_{i\in J}X_i$ and, for $J\subseteq J'\subseteq I$, let $\iota_{JJ'}:X_J\rightarrowtail X_{J'}$ the canonical section. We also put $\iota_i^J:=\iota_{\{i\}J}:X_i\rightarrowtail X_J$, for each $i\in J$, so that, in particular, $\iota_j^I=\iota_j:X_j\longrightarrow X_I=\coprod_{i\in I}X_i$ is the canonical injection into the coproduct, for all $j\in I$. The following result of Category Theory is folklore. We sketch the basic ideas of its proof for the convenience of the reader.

\begin{prop} \label{prop.coproduct-as-directlimit}
Let $\mathcal{A}$ be an additive category with coproducts and  $(X_i)_{i\in I}$ a family of objects in $\mathcal{A}$. Let $\mathcal{J}\subseteq\mathcal{P}(I)$ be a subset that is directed with respect to the inclusion order of $\mathcal{P}(I)$ and satisfies that $\bigcup_{J\in\mathcal{J}}J=I$. Then $[(X_J)_{J\in\mathcal{J}},(\iota_{JJ'}:X_J\rightarrowtail X_{J'})_{J\subseteq J'\text{ }J,J'\in\mathcal{J}}]$ is a direct system in $\mathcal{A}$ that has $(\coprod_{i\in I}X_i, (\iota_J:=\iota_{JI}:X_J\rightarrowtail\coprod_{i\in I}X_i)_{J\in\mathcal{J}})$ as its direct limit.
\end{prop}
\begin{proof} {\bf (Sketch)}
Let $(f_J:X_J\longrightarrow A)_{J\in\mathcal{J}}$ be a family of morphisms such that $f_{J}= f_J\circ \iota_{JJ'}$, whenever $J\subseteq J'$ and $J,J'\in\mathcal{J}$. For each $j\in I$, we choose a $J\in\mathcal{J}$ such that $j\in J$, which is possible since $\bigcup_{J\in\mathcal{J}}J=I$. We put $f_j=f_J\circ\iota_j^J$ and leave to the reader the easy verfication that the definition of $f_j$ does not depend on the choice of $J$. Now, by the universal property of the coproduct, we get a unique morphism $f:\coprod_{i\in I}X_i\longrightarrow A$ such that $f\circ\iota_i=f_i$, for all $i\in I$. For any $J\in\mathcal{J}$ and any $j\in J$, we then have that $f\circ\iota_J\circ\iota_j^J=f\circ\iota_j=f_j=f_J\circ\iota_j^J$. This implies that $f\circ\iota_J=f_J$, for all $J\in\mathcal{J}$, and the uniqueness of $f$ satisfying this property is left as an exercise.
\end{proof}

The following result follows via an adaptation of a famous argument of Lazard \cite[Th\'eor\`eme 1.2]{L}.

\begin{prop}[Generalized Lazard's Trick] \label{prop.Lazard}
Let $\mathcal{A}$ be an additive category with coproducts, let $[(X_i)_{i\in I},(u_{ij}:X_i\to X_j)_{i\leq j})]$ and  $[(Y_\lambda)_{\lambda\in\Lambda},(v_{\lambda\mu}:Y_\lambda\to Y_\mu)_{\lambda\leq\mu}]$ be two direct systems in $\mathcal{A}$ that have a direct limit and let $f:\varinjlim X_i\to\varinjlim Y_\lambda$ be any morphism satisfying the following properties:

$(\dagger)$ For each $j\in I$, there is a $\mu =\mu (j)\in\Lambda$ such that $f\circ u_j$ factors through $v_\mu$, where $u_j:X_j\to\varinjlim X_i$ and $v_\mu :Y_\mu\to\varinjlim Y_\lambda$ are the canonical maps  to the direct limit.

$(\dagger\dagger)$ $v_\mu:Y_\mu\longrightarrow\varinjlim Y_\lambda$ is a monomorphism, for all $\mu\in\Lambda$.

Then  there exists a directed set $\Omega$ and a direct system of morphisms $(g_\omega :X_\omega\to Y_\omega \mid \omega\in\Omega)$ satisfying the following properties:

\begin{enumerate}
\item[(a)] $X_\omega\in\{X_i\text{: }i\in I\}$ and $Y_\omega\in\{Y_\lambda\text{: }\lambda\in\Lambda\}$, for all $\omega\in\Omega$;
\item[(b)] the direct limit of the given direct system of morphisms exists and  $\varinjlim g_\omega :\varinjlim X_\omega\to\varinjlim Y_\omega$ is isomorphic to $f$. 
\end{enumerate}

Moreover, if $\alpha$ is a regular cardinal such that $I$ and $\Lambda$ are $\alpha$-directed, then $\Omega$ may be chosen to be also $\alpha$-directed.
\end{prop}
\begin{proof}
Except for the final assertion concerning $\alpha$, the rest of the proof is identical to the one in \cite[Proposition 9.16]{CCS}, under condition (1) of that result. This latter condition has been replaced here by condition $(\dagger\dagger)$,  that makes the AB5 hypothesis used there unnnecessary. We borrow all the terminology from that mentioned proof. 

Let then assume that $I$ and $\Lambda$ are $\alpha$-directed. We choose exactly the same $\Omega$ of [op.cit], i.e. $\Omega$ consists of the triples $(i,\lambda,g)$ such that $(i,\lambda)\in I\times\Lambda$ and $g:X_i\longrightarrow Y_\lambda$ is a (unique) map such that $v_\lambda\circ g=f\circ u_i$. Let us consider any family [$(i_t,\lambda_t,g_t)]_{t\in T}$ of elements of $\Omega$, where $|T|<\alpha$. Due to the $\alpha$-directed condition of $I$ and $\Lambda$, we can choose $i\in I$ and $\lambda\in\Lambda$ such that $i_t\leq i$ and $\lambda_t\leq\lambda$, for all $t\in T$. By condition $(\dagger)$, there exists a $\mu=\mu (i)\in\Lambda$ together with a map $g:X_i\longrightarrow Y_\mu$ such that $v_\mu\circ g=f\circ u_i$, equivalently, such that $(i,\mu,g)\in\Omega$. Without loss of generality, we can assume that $\lambda\leq\mu$. Then, for each $t\in T$, we have an equality $$v_\mu\circ g\circ u_{i_ti}=f\circ u_i\circ u_{i_ti}=f\circ u_{i_t}=v_{\lambda_t}\circ g_t=v_\mu\circ v_{\lambda_t\mu}\circ g_t.$$ By condition $(\dagger\dagger)$, the map $v_\mu$ is a monomorphism, which implies that $g\circ u_{i_ti}=v_{\lambda_t\mu}\circ g$ and hence that $(i_t,\lambda_t,g_t)\preceq (i,\mu ,g)$. Therefore $\Omega$ is $\alpha$-directed.
\end{proof}

\begin{cor} \label{cor.morphism between coproducts as colimit}
Let  $f:\coprod_{j\in J}X'_j\longrightarrow\coprod_{i\in I}X_i$ be a morphism in $\mathcal{A}$, where $X'_j,X_i\in\text{Small}_\alpha (\mathcal{A})$ for all $j\in J$ and $i\in I$. For each set $S$, let us denote by $\mathcal{P}_\alpha (S)$ the set of subsets $S'\subseteq S$ such that $|S'|<\alpha$, that is an $\alpha$-directed set with the inclusion order. The following assertions hold:

\begin{enumerate}
\item With notation as in the first paragraph of this section, we have direct systems $(X'_\Lambda)_{\Lambda\in\mathcal{P}_\alpha (J)}$ and  $(X_\Upsilon)_{\Upsilon\in\mathcal{P}_\alpha (I)}$ in $\mathcal{A}$ such that $\varinjlim_{\Lambda\in\mathcal{P}_\alpha (J)}X'_\Lambda\cong\coprod_{j\in J}X'_j$ and $\varinjlim_{\Upsilon\in\mathcal{P}_\alpha (I)}X_\Upsilon\cong\coprod_{i\in I}X_i$.
\item There is an $\alpha$-directed system of morphisms $(f_\omega:\hat{X}'_\omega\longrightarrow\hat{X}_\omega )_{\omega\in\Omega}$ satisfying the following properties:

\begin{enumerate}
\item $\hat{X}'_\omega\in\{X'_\Lambda\text{: }\Lambda\in\mathcal{P}_\alpha (J)\}$ and $\hat{X}_\omega\in\{X_\Upsilon\text{: }\Upsilon\in\mathcal{P}_\alpha (I)\}$, for all $\omega\in\Omega$;
\item the direct limit of the given direct system of morphisms exists and $\varinjlim f_\omega :\varinjlim\hat{X}'_\omega\longrightarrow\varinjlim\hat{X}_\omega$ is isomorphic to $f$.
\end{enumerate}
\end{enumerate}
\end{cor}
\begin{proof}
1) Bearing in mind that, for each set $S$, one has that $\bigcup_{S'\in\mathcal{P}_\alpha (S)}S'=S$, this assertion is a direct consequence of Proposition \ref{prop.coproduct-as-directlimit}.

\vspace*{0.3cm}

2) We view $f$ as a morphism $\varinjlim_{\Lambda\in\mathcal{P}_\alpha (J)}X'_\Lambda\longrightarrow\varinjlim_{\Upsilon\in\mathcal{P}_\alpha (I)}X_\Upsilon$ and check that conditions $(\dagger )$ and $(\dagger\dagger )$ of Proposition \ref{prop.Lazard} are satisfied, so that assertion 2 will be a consequence of that propositon. 

Note first that the canonical map $\iota_{\Upsilon'}:X_{\Upsilon'}=\coprod_{i\in\Upsilon'}X_i\longrightarrow\varinjlim_{\Upsilon\in\mathcal{P}_\alpha (I)}X_\Upsilon=\coprod_{i\in I}X_i$ is a section, whence a monomorphism, for all $\Upsilon'\in\mathcal{P}_\alpha (I)$. Therefore condition $(\dagger\dagger )$ holds.   It remains to check property $(\dagger)$. Indeed if $\iota'_{\Lambda'}:X'_{\Lambda'}=\coprod_{j\in\Lambda '}X'_j\longrightarrow\varinjlim_{\Lambda\in\mathcal{P}_\alpha (\Lambda)}X'_\Lambda\cong\coprod_{j\in J}X'_j$ denotes the canonical section, for each $\Lambda'\in\mathcal{P}_\alpha (J)$, then $f\circ\iota'_{\Lambda '}:X'_{\Lambda'}\longrightarrow\varinjlim_{\Upsilon\in\mathcal{P}_\alpha (I)}X_\Upsilon\cong\coprod_{i\in I}X_i$ is a morphism with $\alpha$-small domain. We can then select an $\Upsilon'\in\mathcal{P}_\alpha (I)$ such that $f\circ\iota'_{\Lambda'}$ factors through $\iota_{\Upsilon'}:\coprod_{i\in\Upsilon'}X_i=X_{\Upsilon'}\longrightarrow\varinjlim_{\Upsilon\in\mathcal{P}_\alpha (I)}X_\Upsilon\cong\coprod_{i\in I}X_i$.
\end{proof}

\section{The theorem} \label{sec.The theorem}

We start by stating the following straightforward consequence of \cite[Lemma 2.2, Theorem 2.3 and Proposition 2.5]{N2}.

\begin{prop} (\cite{N2}) \label{prop.Neeman}
Let $\mathcal{D}$ be a well-generated triangulated category and $\tau =(\mathcal{U},\mathcal{V})$ be a t-structure in $\mathcal{D}$ generated by a set of objects $\mathcal{X}$.
Let $\alpha$ be the  smallest of the regular cardinals $\beta$ such that $\mathcal{X}\subseteq\mathcal{D}^\beta$ and $\mathcal{D}$ is $\beta$-compactly generated. The following assertions hold:
\begin{enumerate}
\item Each morphism $Z\longrightarrow U$, where $U\in\mathcal{U}$ and $Z\in\mathcal{D}^\alpha$, factors through an object of $\mathcal{U}\cap\mathcal{D}^\alpha$
\item Each object of $\mathcal{U}$ is the Milnor colimit  of a sequence  $U_0\stackrel{f_1}{\longrightarrow}U_1\longrightarrow \cdots\stackrel{f_1}{\longrightarrow}U_n\longrightarrow...$ such that $U_0$ and all  $\text{cone}(f_n)$ are  coproducts of objects of $\mathcal{U}\cap\mathcal{D}^\alpha$.
\end{enumerate}
\end{prop}

We can now state and have all the ingredients to prove the main result of the paper, which is an extended version of Theorem A in the introduction.

\begin{teor} \label{teor.main result2}
Let $\mathcal{D}$ be a well-generated triangulated category, $\tau =(\mathcal{U},\mathcal{V})$ be a t-structure generated by a set of objects $\mathcal{X}$ and $\mathcal{H}$ be its heart. Suppose that   $\alpha$ is the smallest of the regular cardinals $\beta$ such that $\mathcal{X}\subseteq\mathcal{D}^\beta$ and $\mathcal{D}$ is $\beta$-compactly generated. Let $\mathcal{S}$ be a skeleton of $\mathcal{U}\cap\mathcal{D}^\alpha$. The following assertions hold:

\begin{enumerate}
\item  Each object of $\mathcal{H}$ is an $\alpha$-directed colimit of objects of $H(\mathcal{S})$, where $H:=H_\tau^0:\mathcal{D}\longrightarrow\mathcal{H}$ is the associated cohomological functor. 
\item $H(\mathcal{S})$ consists of $\alpha$-presentable objects and each $\alpha$-presentable object of $\mathcal{H}$ is isomorphic to one of them.
\end{enumerate}
In particular $\mathcal{H}$ is a locally $\alpha$-presentable abelian category. 
\end{teor}

\begin{rem} \label{rem.Saorin-Stovicek result}
When $\alpha=\aleph_0$ in last theorem, one gets  \cite[Theorem 8.31]{SS} since, as shown in the proof of this latter result, the heart of any compactly generated t-structure $\tau$ in a triangulated category with coproducts $\mathcal{D}$ is also the heart of a t-structure with the same aisle in a compactly generated subcategory of $\mathcal{D}$. 
\end{rem}

The proof of Theorem \ref{teor.main result2} will occupy the rest of this section and leans upon some auxiliary results. In the rest of the section we use the hypotheses and  notation of Theorem  \ref{teor.main result2}.

Recall (cf. \cite[Definition 4.2]{SS}) that a subcategory $\mathcal{P}$ of $\mathcal{U}$ is \emph{$t$-generating} when it is precovering and, for each $U\in\mathcal{U}$, there is a triangle $U'\longrightarrow P\stackrel{p}{\longrightarrow}U\stackrel{+}{\longrightarrow}$ such that $U'\in\mathcal{U}$ and $P\in\mathcal{P}$. Such a morphism $p$ will be called here a \emph{$\mathcal{P}$-morphism}. When $\mathcal{P}$ is $t$-generating,  any $\mathcal{P}$-precover of $U$ is a $\mathcal{P}$-morphism (see \cite[Lemma 4.4]{SS}).   Note also that, due to the octahedrom axiom, if $\tilde{U}'\longrightarrow U\stackrel{h}{\longrightarrow}\tilde{U}\stackrel{+}{\longrightarrow}$ is a triangle with its three vertices in $\mathcal{U}$, then $h\circ p$ is also a $\mathcal{P}$-morphism. On the other hand, when $\mathcal{P}$ is closed under coproducts,  any coproduct of $\mathcal{P}$-morphisms is a $\mathcal{P}$-morphism since coproducts of triangles are triangles (see \cite[Proposition 1.2.1 and Remark 1.2.2]{N}) and $\mathcal{U}$ is closed under coproducts.

\begin{lema} \label{lema.t-generating}
$\text{Coprod}(\mathcal{S})$ is a t-generating subcategory of $\mathcal{U}$
\end{lema}
\begin{proof}
We put $\mathcal{P}:=\text{Coprod}(\mathcal{S})$, which  is clearly  precovering and closed under coproducts.  Consider, for any object $U\in\mathcal{U}$, a sequence $U_0\stackrel{f_1}{\longrightarrow}U_1\longrightarrow...\stackrel{f_1}{\longrightarrow}U_n\longrightarrow...$  giving rise to a 
Milnor triangle $$\coprod_{n\in\mathbb{N}}U_n\stackrel{1-\sigma}{\longrightarrow}\coprod_{n\in\mathbb{N}}U_n\longrightarrow U\stackrel{+}{\longrightarrow}$$  as in Proposition \ref{prop.Neeman}.  Suppose that each $U_n$ admits $\mathcal{P}$-morphism $p_n:\hat{S}_n\rightarrow U_n$ for each $n\in\mathcal{N}$. Then $\coprod p_n: \coprod_{n\in\mathbb{N}}\hat{S}_n\longrightarrow\coprod_{n\in\mathbb{N}}U_n$ is a $\mathcal{P}$-morphism. The composition  $$\coprod_{n\in\mathbb{N}}\hat{S}_n\stackrel{\coprod p_n}{\longrightarrow}\coprod_{n\in\mathbb{N}}U_n\stackrel{\text{canonical}}{\longrightarrow}\text{Mcolim}U_n=U$$ is also a $\mathcal{P}$-morphism by the comments preceding this lemma. The proof  is hence reduced to prove our following claim:

\hspace*{1cm} {\it Claim}: For each $n\in\mathbb{N}$ and $U$ a $n$-fold extension of objects of $\mathcal{P}$, there is a $\mathcal{P}$-morphism  $\coprod_{i\in I}S_i\longrightarrow U$

The proof of the claim goes by induction on $n$, the case $n=0$ being clear. Suppose that $n>0$ and consider a triangle $\coprod_{i\in I}S_i\longrightarrow U\stackrel{g}{\longrightarrow} \tilde{U}\stackrel{+}{\longrightarrow},$ where $\tilde{U}$ is a $(n-1)$-fold extension of objects in $\mathcal{P}$ and $S_i\in\mathcal{S}$, for all $i\in I$. By the induction hypothesis, we have another triangle $\tilde{U}'\longrightarrow\coprod_{j\in J}S'_j\stackrel{h}{\longrightarrow}\tilde{U}\stackrel{+}{\longrightarrow}$, where $\tilde{U}'\in\mathcal{U}$ and $S'_j\in\mathcal{S}$ for all $j\in J$.  Taking the homotopy pullback of $g$ and $h$ and drawing the corresponding commutative diagram, we get two triangles $$\tilde{U}'\longrightarrow\hat{U}\stackrel{p}{\longrightarrow} U\stackrel{+}{\longrightarrow}$$ $$\coprod_{i\in I}S_i\longrightarrow\hat{U}\longrightarrow\coprod_{j\in J}S'_j\stackrel{+}{\longrightarrow}$$ Suppose that we are able to prove the existence of a  $\mathcal{P}$-morphism $\coprod_{\lambda\in\Lambda}S''_\lambda\stackrel{q}{\longrightarrow}\hat{U}\stackrel{+}{\longrightarrow}$, where  $S''_\lambda\in\mathcal{S}$, for all $\lambda\in\Lambda$. The comments preceding this lemma tell us that $\coprod_{\lambda\in\Lambda}S''_\lambda\stackrel{p\circ q}{\longrightarrow}U$ is a $\mathcal{P}$-morphism.

The consequence of last paragraph is that, after replacing $U$ by $\hat{U}$ if necessary,  it is enough to prove our claim for the case $n=1$. So in the rest of the proof we assume that $U$ is a ($1-$fold) extension of objects in $\text{Coprod}(\mathcal{S})$ so that, after shift, we get a triangle $$\coprod_{j\in J}S'_j[-1]\stackrel{f}{\longrightarrow}\coprod_{i\in I}S_i\longrightarrow U\stackrel{+}{\longrightarrow},$$ where $S'_j,S_i\in\mathcal{S}$ for all $i\in I$ and $j\in J$. Applying now Corollary \ref{cor.morphism between coproducts as colimit}, we conclude that there is an $\alpha$-directed system of morphisms $(f_w:X'_w\longrightarrow X_w)_{w\in\Omega}$ such that its direct limit exists in $\mathcal{D}$, $f$ is isomorphic to $\varinjlim f_w:\varinjlim X'_w\longrightarrow\varinjlim X_w$ and the $X'_w$ and the $X_w$ are subcoproducts with $<\alpha$ terms of $\coprod_{j\in J}S'_j[-1]$ and $\coprod_{i\in I}S_i$, respectively. In particular, we have that $X'_w,X_w\in\mathcal{D}^\alpha$ and, since $\mathcal{D}^\alpha$ is closed under extensions, we also get that $U_w:=cone(f_w)\in\mathcal{U}\cap\mathcal{D}^\alpha$, for all $w\in\Omega$. 
Moreover, we have a commutative diagram with triangles in rows:

$$\xymatrix{\coprod_{w\in\Omega}X'_w \ar[rr]^{\hspace{0,5 cm}\coprod f_{w}} \ar[d]_{p'} && \coprod_{w\in\Omega}X_{w} \ar[r] \ar[d]^{p} & \coprod_{w\in\Omega}U_{w} \ar[r]^{\hspace{0,5 cm}+} & \\ \varinjlim X'_w \ar[rr]^{\varinjlim f_w} && \varinjlim X_w \ar[r] & U \ar[r]^{+} & }$$

where $p'$ and $p$ are the canonical epimorphisms (=retractions in our case) from the coproduct to the direct limit.
In particular, we have that  $\text{cocone}(p')\in\text{Add}(\mathcal{S})[-1]$ and $\text{cocone}(p)\in\text{Add}(\mathcal{S})$. Applying the dual of Verdier's $3\times 3$ Axiom (see \cite[Lemma 2.6]{May} ), we can complete the last diagram with dotted arrows so that all rows and columns of the diagram are triangles:

$$\xymatrix{\text{cocone}(p') \ar@{-->}[rr] \ar@{-->}[d]&& \text{cocone}(p) \ar@{-->}[r] \ar@{-->}[d] & \text{cocone}(p'') \ar@{-->}[d] \ar[r]^{\hspace{0,5 cm}
+}& \\ \coprod_{w\in \Omega}X'_w \ar[rr]^{\coprod f_w} \ar[d]^{p'} && \coprod_{w\in \Omega}X_w \ar[r] \ar[d]^{p} & \coprod_{w\in \Omega}U_{w} \ar[r]^{\hspace{1 cm}+} \ar@{-->}[d]^{p''} & \\ \varinjlim X'_w\ar[rr]^{\varinjlim f_w} \ar[d]^{+} && \varinjlim X_w \ar[r] \ar[d]^{+} & U \ar[r]^{+} \ar[d]^{+} & \\ &&&&
}$$

From the top horizontal arrow and the fact that $\text{cocone}(p')[1],\text{cocone}(p)\in\text{Add}(\mathcal{S})$ we immediately get that $\text{cocone}(p'')\in\mathcal{U}$.  Our $\mathcal{P}$-morphism for $U$ is the $p''$ in  the right vertical column of the diagram. Indeed, for each $w\in\Omega$, we can take the unique object $\tilde{S}_w\in\mathcal{S}$ that is isomorphic to $U_w$, and  get a triangle as desired $$\text{cocone}(p'')\longrightarrow\coprod_{w\in\Omega}\tilde{S}_{w}\longrightarrow U\stackrel{+}{\longrightarrow}.$$ 

\end{proof}

\begin{lema} \label{lema.objects of H as directed colimits}
The following assertions hold, for each object $M$ of $\mathcal{H}$:

\begin{enumerate}
\item $M\cong\text{Coker}(H(f))$, where the cokernel is taken in $\mathcal{H}$,  for some morphism $f$ in $\text{Coprod}(\mathcal{S})$.
\item There exists an $\alpha$-directed set $\Omega$, a family $(g_{\nu\omega}:Z_\nu\longrightarrow Z_\omega)_{\nu\leq\omega}$ of morphisms in $\mathcal{S}$, where $\nu,\omega\in\Omega$,  such that $[(H(Z_\nu)_{\nu\in\Omega},H((g_{\nu\omega}):H(Z_\nu)\longrightarrow H(Z_\omega))_{\nu\leq\omega}]$ is an $\Omega$-direct sytem in $\mathcal{H}$ and $M$ is isomorphic to its direct limit. 
\end{enumerate}
\end{lema}
\begin{proof}
We again put $\mathcal{P}:=\text{Coprod}(\mathcal{S})$ all through the proof. 

(1) Using Lemma \ref{lema.t-generating},  choose a $\mathcal{P}$-morphism $p:\coprod_{i\in I}S_i\longrightarrow M$, thus leading to a triangle $M[-1]\stackrel{v}{\longrightarrow}U\stackrel{u}{\longrightarrow} \coprod_{i\in I}S_i \stackrel{p}{\longrightarrow}M$, with $U\in\mathcal{U}$. 
Next we take a $\mathcal{P}$-morphism $q:\coprod_{j\in J}S'_j\longrightarrow U$. By taking the homotopy pullback of $q$ and $v$,  we get the following commutative diagram, where columns 1, 2, 4 and rows 2, 3 are triangles and $U',\tilde{U}\in\mathcal{U}$: 

$$\xymatrix{U' \ar@{=}[r] \ar[d] & U' \ar[d] && U'[1] \ar[d]\\ \tilde{U}[-1] \ar[r] \ar[d]& \coprod_{j\in J} S'_j \ar[r]^{f} \ar[d]^{q}& \coprod_{i \in I} S_i \ar@{=}[d] \ar[r] & \tilde{U} \ar[d]\\ M[-1] \ar[d]^{+} \ar[r]_v & U \ar[d]^{+} \ar[r]_u & \coprod_{i \in I} S_i \ar[r]_p & M \ar[d]^{+} \\  & & & }$$

Bearing in mind that $H$ vanishes on $\mathcal{U}[1]$, application of this functor to the right vertical column yields an isomorphism $H(\tilde{U})\cong H(M)=M$. Application to the central row yields an exact sequence $H(\coprod_{j\in J}S'_j)\stackrel{H(f)}{\longrightarrow}H(\coprod_{i\in I}S_i)\longrightarrow H(\tilde{U})\rightarrow 0$.

\vspace*{0.3cm}

(2) Note first  that if $h:S\longrightarrow S'$ is a morphism in $\mathcal{S}$, then $\text{cone}(h)\in\mathcal{U}\cap\mathcal{D}^\alpha$. As a consequence, up to replacement by an isomorphic object, we can assume that $\text{cone}(f)\in\mathcal{S}$. 

Using assertion 1, choose a morphism $f:\coprod_{j\in J}S'_j\longrightarrow\coprod_{i\in I}S_i$ in $\mathcal{P}$ such that $\text{Coker}(H(f))\cong M$. By  Corollary \ref{cor.morphism between coproducts as colimit}, we have an $\alpha$-directed system of morphisms $(f_\omega:\hat{X}'_\omega\longrightarrow\hat{X}_\omega)_{\omega\in\Omega}$ satisfying condition 2 of that corollary, where the direct limits  exist in $\mathcal{D}$. Note that $\hat{X}'_\omega$ and $\hat{X}_\omega$ are coproducts of less than $\alpha$ objects of $\mathcal{S}$, and hence, up to isomorphism, all of them are objects of $\mathcal{\mathcal{S}}$. In particular, as mentioned above, we can assume  that $Z_\omega:=\text{cone}(f_\omega )\in\mathcal{S}$ for each $\omega\in\Omega$. 

 Since  $(f_\omega:\hat{X}'_\omega\longrightarrow\hat{X}_\omega)_{\omega\in\Omega}$ is a direct system of morphisms in $\mathcal{U}$ and $H_{| \mathcal{U}}:\mathcal{U}\longrightarrow\mathcal{H}$ is a left adjoint functor (see \cite[Lemma 3.1]{PS1}), we conclude that there is an exact sequence in $\mathcal{H}$  $$\varinjlim H(\hat{X}'_\omega)\stackrel{\varinjlim H(f_\omega)}{\longrightarrow}\varinjlim H(\hat{X}_\omega)\longrightarrow M\rightarrow 0$$ since $H(f)=H(\varinjlim f_\omega)\cong\varinjlim H(f_\omega)$. But this gives that $M\cong\varinjlim\text{Coker}_\mathcal{H}(H(f_\omega))$ since direct limits are right exact in any cocomplete abelian category. On the other hand, for each $\omega\in\Omega$, we have an exact sequence $H(\hat{X}'_\omega)\stackrel{H(f_\omega)}{\longrightarrow}H(\hat{X}_\omega)\longrightarrow H(Z_\omega)\rightarrow 0$, so that $M\cong\varinjlim  H(Z_\omega)$ and the first paragraph of this step says that, up to isomorphism, $Z_\omega\in\mathcal{S}$. The desired family of morphisms in $\mathcal{S}$ is given as follows. For each $\nu,\omega\in\Omega$ such that $\nu\leq\omega$, let $g_{\nu\omega}:Z_\nu\longrightarrow Z_\omega$ be a fixed choice of a morphism giving rise to a morphism of triangles as follows

$$\xymatrix{\hat{X}_\nu' \ar[r]^{f_{\nu}} \ar[d] & \hat{X}_{\nu} \ar[r] \ar[d] & Z_{\nu} \ar[r]^{+} \ar@{-->}[d]^{g_{\nu \omega}} & \\ \hat{X}'_{\omega} \ar[r]^{f_{\omega}} & \hat{X}_{\omega} \ar[r]& Z_{\omega} \ar[r]^{+} & }$$

where the two left most vertical arrows are the maps in the direct systems $(\hat{X}'_\omega)_{\omega\in\Omega}$ and  $(\hat{X}_\omega)_{\omega\in\Omega}$, respectively. 
Note that, by the universal property of cokernels,  $[(H(Z_\nu)_{\nu\in\Omega},H((g_{\nu\omega}):H(Z_\nu)\longrightarrow H(Z_\omega))_{\nu\leq\omega}]$ is an $\Omega$-direct sytem in $\mathcal{H}$ whose direct limit is isomorphic to $M$. 
\end{proof}

Recall from Proposition \ref{prop.C-modules} that the category $\text{mod}-\mathcal{P}$ is abelian with coproducts  and $\{(-,S)\text{: }S\in\mathcal{S}\}$ is a set of projective $\alpha$-presentable generators. Moreover, by  Lemma \ref{lema.t-generating} and \cite[Theorem 4.5]{SS},  we have a Serre quotient functor $F:\text{mod}-\mathcal{P}\longrightarrow\mathcal{H}$ with a fully faithful right adjoint $G:\mathcal{H}\longrightarrow\text{mod}-\mathcal{P}$ that takes $M\rightsquigarrow\text{Hom}_\mathcal{D}(-,M)_{| \mathcal{P}}$. 


\begin{lema} \label{lem.FP2-alpha}
In the situation of the last paragraph let us put $\mathcal{T}=\text{Ker}(F)$.  Then $\mathcal{T}=\text{Gen}(\mathcal{T}_0)$, where $\mathcal{T}_0$ is the subcategory of $\mathcal{T}$ that consists of the objects $T$ in $\mathcal{T}$ that admit a projective presentation in $\text{mod}-\mathcal{P}$  $$(-,S_2)\longrightarrow (-,S_1)\longrightarrow (-,S_0)\rightarrow T\rightarrow 0, $$ where $S_i\in\mathcal{S}$ and $(-,S_i):=\text{Hom}_\mathcal{P}(-,S_i)$ for $i=0,1,2$.
\end{lema}
\begin{proof}
We apply Proposition \ref{prop.Serre quotient functor}, taking $\mathcal{X}:=\{(-,S)\text{: }S\in\mathcal{S}\}$ as set of generators of $\text{mod}-\mathcal{P}$. The mentioned proposition says that $\mathcal{T}$ is a hereditary torsion class in $\text{mod}-\mathcal{P}$  generated by  the 'cyclic' objects of $\mathcal{T}$, i.e. those objects of $\mathcal{T}$ which are epimorphic image of $(-,S)$, for some $S\in\mathcal{S}$. Imitating the proof of  \cite[Lemma 8.30]{SS}, we just need to prove that if    $S\in\mathcal{S}$ and $M:=(-,S)/N$ is an object of $\mathcal{T}$ then there is an epimorpism $T\twoheadrightarrow M$, where $T\in\mathcal{T}$ admits a projective presentation in  $\text{mod}-\mathcal{P}$ as mentioned in the statement of the lemma. The proof is just an adaptation of that of \cite[Lemma 8.30]{SS}, that we indicate leaving the details to the reader. 

By the proof of \cite[Lemma 4.7]{SS},  $\mathcal{T}$ consists of the $X\in\text{mod}-\mathcal{P}$ for which there exists a triangle $U\longrightarrow P_1\stackrel{g}{\longrightarrow} P_0\stackrel{+}{\longrightarrow}$ in $\mathcal{D}$ such that $X\cong\text{Coker}[(-,g)]$, where $U\in\mathcal{U}$ and $P_1,P_0\in\mathcal{P}:=\text{Coprod}(\mathcal{S})$. For $X=M$ as above, one can then consider a triangle $S[-1]\stackrel{v}{\longrightarrow} U\longrightarrow P\stackrel{g}{\longrightarrow} S$, where $P\in\mathcal{P}$ and $U\in\mathcal{U}$,  such that $M\cong\text{Coker}[(-,g)]$. Since $S[-1]\in\mathcal{D}^\alpha$ Proposition \ref{prop.Neeman} says that $v$ factors through an object $S'\in\mathcal{S}$ and then we get a commutative diagram with triangles in rows:

$$\xymatrix{S[-1] \ar@{=}[d] \ar[r] & S' \ar[r] \ar[d] & S'' \ar[d]\ar[r]^{\rho}  & S \ar@{=}[d] \\ S[-1] \ar[r]_v & U \ar[r] & P \ar[r]_g & S}$$
where $S''$ may be taken in $\mathcal{S}$ since $\mathcal{U}\cap\mathcal{D}^\alpha$ is closed under extensions. We  then have an exact sequence $$(-,S')\longrightarrow (-,S'')\stackrel{(-,\rho)}{\longrightarrow}(-,S)\longrightarrow T\rightarrow 0,$$ where $T\in\mathcal{T}$, and hence there is an epimorphism $T\twoheadrightarrow M$ as desired.
\end{proof}

\begin{lema}  \label{lem.H(S) alpha-presented}
$H(\mathcal{S})$ consists of $\alpha$-presentable objects and  each $\alpha$-presentable object of $\mathcal{H}$ is isomorphic to one of them. 
\end{lema}
\begin{proof}
Recall that the Serre quotient  functor $F:\text{mod}-\mathcal{P}\longrightarrow\mathcal{H}$ satisfies that $F(\text{Hom}_\mathcal{D}(-,S))=H(S)$, for all $S\in\mathcal{S}$ (see \cite[Theorem 4.5]{SS}). Moreover, by Proposition \ref{prop.C-modules}, $\{(-,S)\text{: }S\in\mathcal{S}\}$ is a set of $\alpha$-presentable generators of $\text{mod}-\mathcal{P}$. To prove that $H(\mathcal{S})$ consists of $\alpha$-presentable objects we just need to check that $F$ preserves $\alpha$-presentable objects or, equivalently (see Proposition \ref{prop.preservation-alpha-presented}), that 
$G:\mathcal{H}\longrightarrow\text{mod}-\mathcal{P}$ preserves $\alpha$-filtered colimits. Due to the fully faithful condition of $G$, Lemma \ref{lem.FP2-alpha} and Proposition \ref{prop.Serre quotient functor}, our task reduces to check that $\mathcal{T}_0^{\perp_{0,1}}$ is closed under taking $\alpha$-filtered colimits, where $\mathcal{T}_0$ is as in Lemma \ref{lem.FP2-alpha}. 

Let  $(Y_i)_{i\in I}$ be an $\alpha$-direct system in  $\mathcal{T}_0^{\perp_{0,1}}$ and let $T\in\mathcal{T}_0$, for which we fix a projective presentation $$(-,S_2)\stackrel{(-,f_2)}{\longrightarrow} (-,S_1)\stackrel{(-,f_1)}{\longrightarrow} (-,S_0)\rightarrow T\rightarrow 0, $$ as in Lemma \ref{lem.FP2-alpha}. The fact that $\text{Ext}_{\text{mod}-\mathcal{P}}^k(T,Y_i)=0$ for $k=0,1$ and all $i\in I$ implies that the sequence of abelian groups $$0=(T,Y_i)\longrightarrow ((-,S_0),Y_i))\longrightarrow (S_1,Y_i)\longrightarrow ((-,S_2),Y_i)$$ is exact. Clearly, when $i$ varies, we get an $I$-direct system of exact sequences in $\text{Ab}$. When we apply the direct limit functor, which is exact in $\text{Ab}$, we get an exact sequence

 $$0=\varinjlim (T,Y_i)\longrightarrow\varinjlim ((-,S_0),Y_i))\longrightarrow\varinjlim ((-,S_1),Y_i)\longrightarrow\varinjlim ((-,S_2),Y_i)\hspace*{1cm}(\dagger)$$ 

On the other hand, since the class of $\alpha$-presentable objects is closed under cokernels, we know that $\mathcal{T}_0$ consists of $\alpha$-presentable objects, thus implying in particular that $(T,\varinjlim Y_i)\cong\varinjlim (T,Y_i)=0$. In addition, by the $\alpha$-presentability of $(-,S)$,  the canonical map $\varinjlim ((-,S),Y_i)\longrightarrow ((-,S),\varinjlim Y_i)$ is an isomorphism, for all $S\in\mathcal{S}$. It follows that the sequence $(\dagger)$ above is isomorphic to the sequence  $$0= (T,\varinjlim Y_i)\longrightarrow ((-,S_0),\varinjlim Y_i))\stackrel{(-,f_1)^*}{\longrightarrow} ((-,S_1),\varinjlim Y_i)\stackrel{(-,f_2)^*}{\longrightarrow} ((-,S_2),\varinjlim Y_i).\hspace*{1cm}(\dagger\dagger)$$ In particular, this latter sequence is also exact and we have that  $\text{Ext}_{\text{mod}-\mathcal{P}}^1(T,\varinjlim Y_i)=\frac{\text{Ker} (-,f_2)^*}{\text{Im}(-,f_1)^*}=0$. Therefore $\varinjlim Y_i\in\mathcal{T}_0^{\perp_{0,1}}$ as desired. 

It remains to prove that any $\alpha$-presentable object $M$ of $\mathcal{H}$ is isomorphic to $H(S)$, for some $S\in\mathcal{S}$. Due to Lemma \ref{lema.objects of H as directed colimits}, we know that $M\cong\varinjlim H(S_i)$, for some $\alpha$-direct system $[(H(S_i))_{i\in I},(H(f_{ij}:H(S_i)\longrightarrow H(S_j))_{i\leq j}]$ in $H(\mathcal{S})$. Due to the $\alpha$-presentability of $M$, a standard argument shows that $M$ is isomorphic to a direct summand of some $H(S_j)$. The proof will be finished once we prove that $H(\mathcal{S})\cong H(\mathcal{U}\cap\mathcal{D}^\alpha)$ is closed under taking direct summands. This follows from an easy adaptation of the previous to last paragraph of the proof of \cite[Therem 8.31]{SS}.  $\mathcal{U}_0$ should be replaced by $\mathcal{S}$ here and the finite subset $J\subseteq I$  there by a subset of cardinality $|J|<\alpha$.
\end{proof}

{\bf Proof of Theorem \ref{teor.main result2}:} 

It is a direct consequence of Lemmas \ref{lema.objects of H as directed colimits} and \ref{lem.H(S) alpha-presented}.

\end{document}